\documentclass[12pt]{amsart}

\usepackage{amssymb}
\usepackage{amscd}

\title[Algebraic curves admitting the same Galois closure]{Algebraic curves admitting the same Galois closure for two projections} 

\author{Satoru Fukasawa} 
\address{Department of Mathematical Sciences, Faculty of Science, Yamagata University, Kojirakawa-machi 1-4-12, Yamagata 990-8560, Japan} 
\email{s.fukasawa@sci.kj.yamagata-u.ac.jp} 
\thanks{The first author was partially supported by JSPS KAKENHI Grant Numbers 16K05088 and 19K03438.}

\author{Kazuki Higashine} 
\address{Graduate School of Science and Engineering, Yamagata University, Kojirakawa-machi 1-4-12, Yamagata 990-8560, Japan}
\email{s182102d@st.yamagata-u.ac.jp}

\author{Takeshi Takahashi}
\address{Education Center for Engineering and Technology, Faculty of Engineering, Niigata University, Niigata 950-2181, Japan} 
\email{takeshi@eng.niigata-u.ac.jp}
\thanks{The third author was partially supported by JSPS KAKENHI Grant Numbers 16K05094 and 19K03441.} 


\newtheorem{theorem}{Theorem}

\newtheorem{corollary}{Corollary}

\theoremstyle{definition}
\newtheorem{remark}{Remark}

\begin{document}
\begin{abstract} 
A criterion for the existence of a plane model of an algebraic curve such that the Galois closures of projections from two points are the same is presented. 
As an application, it is proved that the Hermitian curve in positive characteristic coincides with the Galois closures of projections of some plane curve from some two non-uniform points. 
\end{abstract}

\maketitle 

\section{Introduction} 
Let $C \subset \mathbb{P}^2$ be an irreducible plane curve of degree $d \ge 4$ over an algebraically closed field $k$ of characteristic $p \ge 0$, and let $k(C)$ be its function field. 
We consider the projection $\pi_P: C \dashrightarrow \mathbb{P}^1$ from a point $P \in \mathbb{P}^2$.
If $\pi_P$ is separable,   
then the Galois closure of $k(C)/\pi_P^*k(\mathbb{P}^1)$ is denoted by $L_P$, and the Galois group is denoted by $G_P$.  
Pirola and Schlesinger said that $P$ is {\it uniform}, if $G_P$ is the full symmetric group (see \cite{pirola-schlesinger}). 
It follows from results of Yoshihara and Pirola--Schlesinger that there exist only finitely many non-uniform points $P \in \mathbb{P}^2$, at least when $p=0$ and $C$ is smooth (see \cite{pirola-schlesinger, yoshihara}). 
The following problem is natural: {\it for different points $P_1$ and $P_2 \in \mathbb{P}^2$, when does $L_{P_1}=L_{P_2}$ hold?} 
However, it has never been considered, except for the case where $L_{P_i}=k(C)$ for $i=1, 2$. 

In this article, we settle this problem. 
Let $X$ be a smooth projective curve. 
For a finite subgroup $H$ of ${\rm Aut}(X)$ and a point $Q \in X$, the quotient map is denoted by $f_H: X \rightarrow X/H$ and the image $f_H(Q)$ is denoted by $\overline{Q}^H$.
We show the following theorem. 

\begin{theorem} \label{maincor}
Let $H$, $G_1$, and $G_2 \subset {\rm Aut}(X)$ be finite subgroups such that $H \subset G_1 \cap G_2$, and let $P_1$ and $P_2 \in X$. 
Then, five conditions
\begin{itemize}
\item[(a)] $X/{G_1} \cong \Bbb P^1$ and $X/{G_2} \cong \Bbb P^1$,    
\item[(b)] $G_1 \cap G_2=H$, 
\item[(c)] $H' \vartriangleleft G_i, \ H' \subset H\Rightarrow H'=\{1\}$, for $i=1, 2$, 
\item[(d)] $\sum_{h \in H}h(P_1)+\sum_{\sigma \in G_1} \sigma (P_2)=\sum_{h \in H}h(P_2)+\sum_{\tau \in G_2} \tau (P_1)$, and
\item[(e)] $HP_1 \ne HP_2$
\end{itemize}
are satisfied, if and only if there exists a birational embedding $\varphi: X/H \rightarrow \mathbb P^2$ of degree $(G_1:H)+1$ such that $\varphi(\overline{P_1}^H)$ and $\varphi(\overline{P_2}^H)$ are different smooth points of $\varphi(X/H)$, and $L_{\varphi(\overline{P_i}^H)}=k(X)$ and $G_{\varphi(\overline{P_i}^H)}=G_i$ for $i=1, 2$. 

\end{theorem}

Since any plane curve is a quotient curve of the smooth model of the Galois closure at each point, all plane curves $C$ with two different smooth points $P_1$ and $P_2 \in C$ such that $L_{P_1}=L_{P_2}$ are described completely in Theorem \ref{maincor}.  

In Section 2, we prove a generalization of Theorem \ref{maincor}, to understand the proof in more general setting. 
In Section 3, we show the following theorem for the Hermitian curve, which may be surprising.

\begin{theorem} \label{hermitian} 
Let $p>0$, $q$ be a power of $p$, and let a positive integer $m$ divide $q^2-1$. 
The Hermitian curve defined by 
$$ X^qZ+XZ^q-Y^{q+1}=0 $$  
is denoted by $X$. 
Then, there exists a plane curve $C \subset \mathbb{P}^2$ of degree $q^3+1$ and different smooth points $P_1$ and $P_2$ exist for $C$ such that $L_{P_1}=L_{P_2}=k(X)$ and $G_{P_i} \cong N_1 \rtimes C_m$ for $i=1, 2$, where
$N_1$ is a Sylow $p$-group of ${\rm Aut}(X)$ and $C_m$ is a cyclic group of order $m$.      
In particular, points $P_1$ and $P_2$ are not uniform. 
\end{theorem}

\section{Proof of the main theorem} 
In this section, we prove Theorem \ref{maincor} and its generalization. 
If $H$ is a normal subgroup of a subgroup $G \subset {\rm Aut}(X)$, then there exists a natural homomorphism $G \rightarrow {\rm Aut}(X/H); \sigma \mapsto \overline{\sigma}^H$, where $\overline{\sigma}^H$ corresponds to the restriction $\sigma^*|_{k(X)^H}$.  
The image is denoted by $\overline{G}^H$, which is isomorphic to $G/H$. 
The following is a generalization of Theorem \ref{maincor}, since the case where $N_1=N_2=\{1\}$ implies Theorem \ref{maincor}. 

\begin{theorem} \label{main} 
Let $N_1$, $N_2$, $H$, $G_1$, and $G_2 \subset {\rm Aut}(X)$ be finite subgroups such that $N_i \subset H \subset G_i$ and $N_i \vartriangleleft G_i$ for $i=1, 2$, and let $P_1$ and $P_2 \in X$. 
Then, five conditions
\begin{itemize}
\item[(a)] $X/{G_1} \cong \Bbb P^1$ and $X/{G_2} \cong \Bbb P^1$,    
\item[(b)] $G_1 \cap G_2=H$, 
\item[(c)] $H' \vartriangleleft G_i$, $N_i \subset H' \subset H$ $\Rightarrow$ $H'=N_i$, for $i=1, 2$, 
\item[(d)] $\sum_{h \in H}h(P_1)+\sum_{\sigma \in G_1} \sigma (P_2)=\sum_{h \in H}h(P_2)+\sum_{\tau \in G_2} \tau (P_1)$, and
\item[(e)] $HP_1 \ne HP_2$
\end{itemize}
are satisfied, if and only if there exists a birational embedding $\varphi: X/H \rightarrow \mathbb P^2$ of degree $(G_1:H)+1$ such that $\varphi(\overline{P_1}^H)$ and $\varphi(\overline{P_2}^H)$ are different smooth points of $\varphi(X/H)$, and $L_{\varphi(\overline{P_i}^H)}=k(X)^{N_i}$ and $G_{\varphi(\overline{P_i}^H)}=\overline{G_i}^{N_i}$ for $i=1, 2$. 
\end{theorem}

\begin{proof}
We consider the only-if part. 
By condition (e), $\overline{P_1}^H \ne \overline{P_2}^H$. 
Let $D$ be the divisor as in condition (d). 
Since 
$$\sum_{\sigma \in G_1}\sigma(P_2)=\sum_{H\sigma \in H \setminus G_1}\sum_{h \in H} h\sigma(P_2), $$
it follows that 
$$(f_H)_*D=|H|\left(\overline{P_1}^H+\sum_{H\sigma \in H\setminus G_1}\overline{\sigma(P_2)}^H\right)=|H|\left(\overline{P_2}^H+\sum_{H\tau \in H \setminus G_2}\overline{\tau(P_1)}^H\right)$$
as divisors on $X/H$. 
Therefore, 
$$ \overline{D}^H:=\overline{P_1}^H+\sum_{H\sigma \in H\setminus G_1}\overline{\sigma(P_2)}^H=\overline{P_2}^H+\sum_{H\tau \in H \setminus G_2}\overline{\tau(P_1)}^H. $$
Let $q_i: X/H \rightarrow \mathbb{P}^1$ be the morphism induced by the extension $k(X)^H/k(X)^{G_i}$. 
Note that 
$$f_H^*(\overline{Q}^H)=\sum_{h \in H}h(Q) \ \mbox{ and } \ (q_i \circ f_H)^*(q_i(\overline{Q}^H))=\sum_{\sigma \in G_i}\sigma(Q)$$ for each point $Q \in X$ (see, for example, \cite[III.7.1, III.7.2, III.8.2]{stichtenoth}). 
It follows that $\overline{D}^H-\overline{P_1}^H$ coincides with the pull-back $q_1^*(q_1(\overline{P_2}^H))$, since $(f_H)^*(\overline{D}^H-\overline{P_1}^H)=\sum_{\sigma \in G_1}\sigma(P_2)=(f_H)^*(q_1^*(q_1(\overline{P_2}^H))$, and $(f_H)_*(f_H)^*(D')=|H|D'$ for any divisor $D' \in {\rm Div}(X/H)$. 
Let $f$ and $g \in k(X)^H$ be generators of $k(X)^{G_1}$ and $k(X)^{G_2}$ such that $(f)_{\infty}=\overline{D}^H-\overline{P_1}^H$ and $(g)_{\infty}=\overline{D}^H-\overline{P_2}^H$, by (a), where $(f)_{\infty}$ is the pole divisor of $f$.   
Then, $f, g \in \mathcal{L}(\overline{D}^H)$. 
Let $\varphi: X/H \rightarrow \Bbb P^2$ be given by $(f:g:1)$. 
To prove that $\varphi$ is birational onto its image, we show that $k(X)^H=k(f, g)$. 
Since $k(X)/k(f)$ is Galois, there exists a subgroup $H_1$ of $G_1$ such that $H_1={\rm Gal}(k(X)/k(f, g))$. 
Similarly, there exists a subgroup $H_2$ of $G_2$ such that $H_2={\rm Gal}(k(X)/k(f, g))$. 
Since $G_1 \cap G_2=H$ by condition (b), $H_1=H_2=H$. 
The morphism $\varphi$ is birational onto its image. 
The sublinear system of $|\overline{D}^H|$ corresponding to $\langle f, g, 1\rangle$ is base-point-free, since ${\rm supp}(\overline{D}^H) \cap {\rm supp}((f)+\overline{D}^H)=\{\overline{P_1}^H\}$ and ${\rm supp}(\overline{D}^H) \cap {\rm supp}((g)+\overline{D}^H)=\{\overline{P_2}^H\}$. 
Therefore, $\deg \varphi(X/H)=\deg \overline{D}^H$, and the morphism $(f:1)$ (resp. $(g:1)$) coincides with the projection from the smooth point $\varphi(\overline{P_1}^H) \in \varphi(X/H)$ (resp. $\varphi(\overline{P_2}^H) \in \varphi(X/H)$). 
The Galois closure of $k(X)^H/k(X)^{G_i}$ coincides with $k(X)^{N_i}$, by condition (c).

We consider the if part. 
Since $\pi_{\varphi(\overline{P_i}^H)}^*k(\mathbb{P}^1)=(k(X)^{N_i})^{\overline{G_i}^{N_i}}=k(X)^{G_i}$, condition (a) is satisfied. 
Since $k(X)^{N_i}$ is the Galois closure of $k(X)^H/k(X)^{G_i}$, conditions (c) is satisfied. 
By the assumption, $H \subset G_1 \cap G_2$. 
To prove (b), we take a suitable system of coordinates so that $\varphi(\overline{P_1}^H)=(0:1:0)$ and $\varphi(\overline{P_2}^H)=(1:0:0)$. 
Then, $k(X)^{G_1}=k(x)$ and $k(X)^{G_2}=k(y)$. 
For $\sigma \in G_1 \cap G_2$, $\sigma^*(x)=x$ and $\sigma^*(y)=y$. 
Since $k(X)^H=k(x,y)$, $\sigma \in H$. 
Condition (b) is satisfied. 
Since $\varphi(\overline{P_1}^H) \ne \varphi(\overline{P_2}^H)$, condition (e) is satisfied. 
Let $\overline{D}^H$ be the divisor induced by the intersection of $\varphi(X/H)$ and the line $\ell:=\overline{\varphi(\overline{P_1})\varphi(\overline{P_2})}$, where $\overline{\varphi(\overline{P_1})\varphi(\overline{P_2})}$ is the line passing through points $\varphi(\overline{P_1})$ and $\varphi(\overline{P_2})$. 
We can consider the line $\ell$ as a point in the images of $\pi_{\varphi(\overline{P_1}^H)}\circ\varphi$ and $\pi_{\varphi(\overline{P_2}^H)}\circ\varphi$. 
Since $\overline{P_2}^H \in \varphi^{-1}(\ell)$, 
$$ (\pi_{\varphi(\overline{P_1}^H)} \circ \varphi)^*\ell=\sum_{H\sigma \in H\setminus G_1}\overline{\sigma(P_2)}^H. $$
Since this divisor coincides with $\overline{D}^H-\overline{P_1}^H$, it follows that 
$$ \overline{D}^H=\overline{P_1}^H+\sum_{H\sigma \in H\setminus G_1}\overline{\sigma(P_2)}^H=\overline{P_2}^H+\sum_{H\tau \in H\setminus G_2}\overline{\tau(P_1)}^H. $$
Considering the divisor $f_H^*(\overline{D}^H)$, condition (d) is satisfied.  
\end{proof}

\begin{remark}
A similar result holds for ``outer'' points. 
In this case, we consider a $6$-tuple $(G_1, G_2, H, N_1, N_2, Q)$ with $Q \in X$ such that $G_1$, $G_2$, $H$, $N_1$ and $N_2$ satisfy conditions (a), (b) and (c), and $\sum_{\sigma \in G_1}\sigma(Q)=\sum_{\tau \in G_2}\tau(Q)$ holds. 
\end{remark}

\section{Examples}
In this section, we assume that the characteristic $p$ is positive and $q$ is a power of $p$. 
The finite field of $q^2$ elements is denoted by $\mathbb{F}_{q^2}$. 
We consider the Hermitian curve $X \subset \mathbb{P}^2$ of degree $q+1$, which is defined by $X^qZ+XZ^{q}-Y^{q+1}=0$. 

\begin{proof}[Proof of Theorem \ref{hermitian}]
Let $P_1=(1:0:0)$ and let $P_2=(0:0:1) \in X$. 
We consider subgroups 
$$N_1=\{\sigma_{a, b}: (X:Y:Z) \mapsto (X+a^qY+b Z:Y+a Z:Z) 
 \ | \ a, b \in \mathbb{F}_{q^2}, b^q+b=a^{q+1} \}, $$
$$N_2=\{(X:Y:Z) \mapsto (X: a X+Y: b X+a^q Y+Z) 
 \ | \ a, b \in \mathbb{F}_{q^2}, b^q+b=a^{q+1} \}, 
$$
and 
$$ C_{q^2-1}=\{\eta_c: (X: Y: Z) \mapsto (c^{q+1}X: c Y :Z) \ | \ c \in \mathbb{F}_{q^2} \setminus \{0\} \} $$
of ${\rm Aut}(X)$. 
Let $H$ be a subgroup of $C_{q^2-1}$ and let $G_i=\langle N_i, H \rangle \cong N_i \rtimes H$. 
We prove that conditions (a), (b), (c), (d) and (e) in Theorem \ref{maincor} are satisfied for $5$-tuple $(G_1, G_2, H, P_1, P_2)$. 
Since it is known that $X/N_i \cong \mathbb{P}^1$, by L\"{u}roth's theorem, condition (a) is satisfied. 
Since $G_1 \cap G_2 \subset \{\sigma \in G_1 \ | \ \sigma(P_2)=P_2\}=H$, condition (b) is satisfied. 
Since $(\sigma_{a, b}^{-1} \eta_{c} \sigma_{a, b})^*y=c y+a(c-1)$ for $\sigma_{a, b} \in N_1$ and $\eta_c \in H$, $H$ does not contain a normal subgroup of $G_1$ other than $\{1\}$.  
Condition (c) is satisfied. 
It is well known that the cardinality of the set $X(\mathbb{F}_{q^2})$ of all $\mathbb{F}_{q^2}$-rational points of $X$ is equal to $q^3+1$, and $N_i$ acts on the set $X(\mathbb{F}_{q^2}) \setminus \{P_i\}$ transitively for $i=1, 2$. 
Since 
$$ \sum_{h \in H}h(P_1)+\sum_{\sigma \in G_1}\sigma(P_2)=\sum_{Q \in X(\mathbb{F}_{q^2})} |H|Q=\sum_{h \in H}h(P_2)+\sum_{\tau \in G_2}\tau(P_1),  $$
condition (d) is satisfied. 
Since $HP_1=\{P_1\} \ne \{P_2\}=HP_2$, condition (e) is satisfied. 
The proof of Theorem \ref{hermitian} is completed. 
\end{proof}

If $ms=q+1$ and $c$ is a primitive $(q^2-1)$-th root of unity, then the subgroup $C_m$ of $C_{q^2-1}$ of order $m$ is generated by $\eta_{c^{s(q-1)}} \in C_{q^2-1}$. 
Since $\eta_{c^{s(q-1)}}^*(x)=x$ and $\eta_{c^{s(q-1)}}^*(y^m)=y^m$, the quotient curve $X/C_m$ has a plane model defined by $x^q+x=y^s$. 
Therefore, the following holds. 

\begin{corollary} 
Let $ms=q+1$, and let $X$ be the Hermitian curve of degree $q+1$.  
Then, for the curve $x^q+x=y^s$, there exist a plane model $C$ of degree $q^3+1$ and different smooth points $P_1$ and $P_2 \in C$ such that $L_{P_i} \cong k(X)$ and $G_{P_i} \cong N_1 \rtimes C_m$ for $i=1, 2$.     
\end{corollary} 

\begin{remark}
A result similar to Theorem \ref{hermitian} holds for the rational, Suzuki or Ree curve (see \cite[Sections 12.2 and 12.4]{hkt} for the properties of the Suzuki or Ree curves).  
\end{remark}

\end{document}